\documentclass{gSTA2e}

\usepackage{color}
\usepackage{soul}
\usepackage{todonotes}

\usepackage{epstopdf}
\usepackage{subfigure}

\theoremstyle{plain}
\newtheorem{theorem}{Theorem}[section]
\newtheorem{corollary}[theorem]{Corollary}
\newtheorem{lemma}[theorem]{Lemma}

\theoremstyle{definition}

\theoremstyle{remark}

\begin{document}



\title{Nonuniform Berry-Esseen bounds for martingales with applications to statistical estimation    \\  \footnotesize{\textit{Dedicated to Paul Doukhan on his sixtieth birthday}} }

\author{
\name{X. Fan\textsuperscript{a,b,}$^\ast$,
I. Grama\textsuperscript{c}
and Q. Liu\textsuperscript{c,}$^{\ast}$\thanks{$^\ast$Fan and Liu are both corresponding authors. }\thanks{$^\ast$Email: fanxiequan@hotmail.com (X. Fan), quansheng.liu@univ-ubs.fr (Q. Liu)}}
\affil{\textsuperscript{a}Center for Applied Mathematics, Tianjin University, 300072 Tianjin,  China. \\
 \textsuperscript{b}Regularity Team, Inria, France; \\
\textsuperscript{c}Universit\'{e} de Bretagne-Sud, LMBA, UMR CNRS 6205,
 Campus de Tohannic, 56017 Vannes, France
}
\received{v5.0 released February 2015}
}

\maketitle

\begin{abstract}
We establish nonuniform Berry-Esseen bounds for martingales under the conditional Bernstein condition. These bounds imply
Cram\'{e}r type large deviations for moderate $x$'s, and are of exponential decay rate as de la Pe\~{n}a's inequality when $x\rightarrow \infty$. Statistical applications associated with linear regressions and   self-normalized large deviations   are also provided.
\end{abstract}


\begin{classcode}  60G42;
 60F05;
 60F10;
 60E15;
 62E20;
62J05
\end{classcode}

\section{Introduction}
\textcolor{black}{Let $n\geq 1$ be an integer and} $(\xi_{i})_{i=1,...,n}$ be a sequence of independent
random variables with zero means and finite variances \textcolor{black}{on a
 probability space $(\Omega ,\mathcal{F},\mathbf{P})$}.
 Let $S_n=\sum_{i=1}^n\xi_i.$
Without loss of generality, we assume that $\mathbf{E}[S_n^2]=1,$
\textcolor{black}{where $\mathbf{E}$ is the expectation corresponding to $\mathbf{P}.$}
Nonuniform normal approximation bounds have been first obtained by Esseen (1945) for
identically distributed random variables with finite third moments. These were improved
to $C\, n\, \mathbf{E} [|\xi_1|^{3} ] /(1+|x|^3)$ by Nagaev (1965), where, throughout this paper, $C$ stands for an \textcolor{black}{absolute}  constant with possibly different values in different places.
Bikelis (1966) generalized Nagaev's result to the sums of non-identically distributed random variables with moments of order larger than $2$.
Suppose that there exists a constant $\delta \in (0, 1]$ such that $\mathbf{E} [|\xi_i|^{2+\delta} ]< \infty$ for all $i \in [1,n].$
\textcolor{black}{Bikelis (1966) (cf.\ also Petrov (1975), p.132)} has established the following nonuniform Berry-Esseen bound: for all $x \in \mathbf{R},$
\textcolor{black}{\begin{eqnarray}
\Big|\mathbf{P}(S_n\leq x) - \Phi(x)\Big| \leq  C\,\frac{\sum_{i=1}^n\mathbf{E}[|\xi_i|^{2+\delta}]}{(1+|x|)^{2+\delta}},\label{bei}
\end{eqnarray}
where} $\Phi(x)$ is the standard normal distribution function.
 Note that Bikelis'  bound implies the Berry-Esseen bound
\begin{eqnarray}
D(S_n)  := \sup_{x \in \mathbf{R}} \Big|\mathbf{P}(S_n\leq x) - \Phi(x) \Big| \leq C\ \sum_{i=1}^n\mathbf{E}[|\xi_i|^{2+\delta}],
\end{eqnarray}
which is known to be optimal.
On the other hand
the  bound (\ref{bei}) decays in the best possible polynomial rate $1/|x|^{2+\delta}$ when $|x|\rightarrow \infty.$
For random variables without assuming the existence of moments of order lager than $2$,
\mbox{Bikelis} type bound has been established by Chen and Shao (2001) via Stein's
method: they showed that, for all $x \in \mathbf{R},$
\begin{eqnarray}\label{csi}
&&\Big|\mathbf{P}(S_n\leq x) - \Phi(x) \Big|  \ \leq \ C \sum_{i=1}^n \bigg( \frac{\mathbf{E}[\xi_i^2\mathbf{1}_{\{|\xi_i|>1+|x|\}}]}{(1+|x|)^2} + \frac{\mathbf{E}[|\xi_i|^3\mathbf{1}_{ \{|\xi_i| \leq 1+|x| \}}] }{(1+|x|)^3}\bigg).
\end{eqnarray}
\textcolor{black}{It is easy to see that (\ref{csi}) implies (\ref{bei}), since
$\mathbf{E}[\xi_i^2\mathbf{1}_{\{|\xi_i|>1+|x|\}}]\leq \mathbf{E}[|\xi_i|^{2+\delta}] / (1+|x|)^\delta$
and
$\mathbf{E}[|\xi_i|^3\mathbf{1}_{\{|\xi_i|\leq 1+|x|\}}]   \leq \mathbf{E}[|\xi_i|^{2+\delta}]   (1+|x|)^{1-\delta}$
for $\delta \in (0,1].$}

Now we consider the case of martingales. The  generalization of Bikelis' bound (\ref{bei}) can be found in   Hall and Heyde (1980, 1981) and Haeusler and Joos (1988b). Assume that $(\xi_{i}, \mathcal{F}_{i})_{i=1,...,n}$ is a sequence of martingale differences defined on some probability space
$(\Omega ,\mathcal{F},\mathbf{P})$. Here $\xi_{i}$'s may depend on $n$.  Denote by $\left\langle S\right\rangle_n=\sum_{i=1}^n\mathbf{E}[\xi _i^2|\mathcal{F} _{i-1}]$ the quadratic characteristic of $S_n$. Haeusler and Joos  proved that if  $\mathbf{E} [|\xi_i|^{2+\delta} ]< \infty$ for some $\delta > 0$ and for all $i \in [1,n],$ then there exists a constant $C_\delta$, depending only on $\delta$, such that, for all $x \in \mathbf{R},$
\begin{eqnarray}
  \Big|\mathbf{P}(S_n\leq x) - \Phi(x) \Big|  \leq \ C_\delta \ \bigg( \sum_{i=1}^n\mathbf{E} [ |\xi_i|^{2+\delta}] +  \mathbf{E} [|\langle S\rangle_{n}-1|^{1+\delta/2}] \bigg)^{1/(3+\delta)} \frac{1}{1+|x|^{2+\delta}}; \label{cdsi}
\end{eqnarray}
see also  Hall and Heyde (1980, 1981) with the larger factor $\frac{1}{1+|x|^{4(1+\delta/2)^2/(3+\delta)}}$ replacing $\frac{1}{1+|x|^{2+\delta}}$ of (\ref{cdsi}).
Similar to Bikelis' bound, inequality (\ref{cdsi}) implies the following Berry-Essen bound for martingales:
\begin{eqnarray}\label{cdsdv}
D(S_n) \ \leq \ C_\delta \ \bigg( \sum_{i=1}^n\mathbf{E} [ |\xi_i|^{2+\delta} ]+  \mathbf{E} [|\langle S\rangle_{n}-1|^{1+\delta/2} ]\bigg)^{1/(3+\delta)}.
\end{eqnarray}
Moreover, Haeusler (1988a) showed that the Berry-Esseen bound (\ref{cdsdv}) is the best possible under the stated conditions.
The bound (\ref{cdsi}) also decays with the best possible polynomial rate $1/|x|^{2+\delta}$ when $|x|\rightarrow \infty.$

Apart the nonuniform Berry-Esseen bounds with polynomially decaying rate there are few nonuniform bounds with exponentially decaying rate. The only reference we aware of is Joss (1991) for bounded martingale differences. We refer also to Ra\v{c}kauskas (1990,1995), Grama (1997) and Grama and Haeusler (2000, 2006) where moderate deviations have been obtained.

Assume the following martingale version of Bernstein's conditions,
where it is assumed that the martingale differences have moments of all orders:
\begin{description}
\item[(A1)]  There exists some positive number  $\epsilon \in (0, 1/2]$ such that
\[
\Big|\mathbf{E}[\xi_{i}^{k}  | \mathcal{F}_{i-1}] \Big| \leq \frac12 k!\epsilon^{k-2} \mathbf{E}[\xi_{i}^2 | \mathcal{F}_{i-1}] \ \ \ a.s.  \ \textrm{for all}\ k\geq 2\   \textrm{and all} \ 1\leq i\leq n.
\]
\item[(A2)] There exists a nonnegative  number $\delta \in [0,1]$ such that
\[
 |\langle S\rangle_{n}-1| \leq \delta^2  \ \  a.s. \ \ \  \ \ \ \ \ \ \ \ \ \ \ \ \ \ \ \ \ \ \ \ \ \ \ \ \ \ \ \ \ \ \ \ \ \ \ \ \ \ \ \ \ \ \ \ \ \ \ \ \ \
\]
\end{description}
Here $\epsilon$ and $\delta$ usually depend on $n$ such that $\epsilon  \rightarrow 0, \delta \rightarrow 0$ as $n\rightarrow \infty$.
In particular, when $(\xi _i)_{i=1,...,n}$ are independent, condition (A2) is satisfied with $\delta=0$ for normalized $S_n$ and condition (A1) is known as the Bernstein condition. If $(\xi _i)_{i=1,...,n}$ are also identically distributed, then (A1) holds with  $\epsilon=\frac {C} {\sqrt{n}} $ as $n\rightarrow \infty.$

Under condition (A1), de la Pe\~{n}a (1999) has obtained the following martingale version of Bennett's inequality (1962) (see also Bernstein (1927)),  for all $x,v >0$,
\begin{eqnarray}
\mathbf{P}(S_n> x, \langle S\rangle_{n}\leq v^2 ) &\leq&  \exp\left\{-\frac{x^2}{ v^2+\sqrt{1+2 x\epsilon/v^2   } +  x \epsilon  }\right\}  \label{fsv}  \\
&\leq&  \exp\left\{-\frac{x^2}{ 2( v^2 + x\epsilon ) }\right\}. \label{scv}
\end{eqnarray}
\textcolor{black}{Recent improvements of \eqref{fsv} are given in Theorem 3.14 of Bercu, Deylon and Rio (2015)
and Fan, Grama and Liu (2015b).}
We refer to Shorack and Wellner (1986) and van de Geer (1995) for  inequality (\ref{scv}).
Moreover, if, in addition, condition (A2) holds, then de la Pe\~{n}a's inequality (\ref{fsv}) implies that, for all $x\geq 0,$
\begin{eqnarray}
\mathbf{P}(S_n> x ) \leq \exp\left\{-\frac{\, \widehat{x}^2}{ 2 }\right\}, \label{fin2}
\end{eqnarray}
where
\begin{eqnarray}\label{fbxr000}
\widehat{x}   = \frac{2x/\sqrt{1+\delta^2}}{1+\sqrt{1+2x\epsilon/(1+\delta^2)}}.
\end{eqnarray}
Since $\widehat{x}\rightarrow x$ as $\max\{\epsilon,
\delta\}\rightarrow 0,$ the bound (\ref{fin2}) is exponentially decaying with rate $\exp\{-x^2/2\}$ when $\max\{\epsilon, \delta\}\rightarrow 0$.
Thus, the  bound (\ref{fin2}) is tight.
By considering the martingale differences $(-\xi_{i}, \mathcal{F}_{i})_{i=1,...,n}$,  the bound (\ref{fin2}) holds equally on tail probabilities
$\mathbf{P}(S_n< - x),\ x\geq 0.$
Due to this fact, we expect to establish a nonuniform Berry-Esseen bound with exponentially decaying rate as in (\ref{fin2}) when $|x|\rightarrow \infty$.

Our main result is the following  nonuniform Berry-Esseen bound for martingales.
Assume   conditions (A1) and (A2). Then, for all $x \geq 0$,
\begin{eqnarray}\label{fgs1}
 && \left|\frac{}{} \mathbf{P}(S_n \leq x )-  \Phi\left( x\right) \right|     \leq   C\,\Big( 1+ x^2 \Big)\Big( \epsilon|\log \epsilon| +\frac{\delta }{1 +|x| }\Big) \exp\left\{-\frac{\widehat{x}^2}{2}\right\} ,
\end{eqnarray}
where $\widehat{x}$ is defined by (\ref{fbxr000}).
To prove (\ref{fgs1}), we need the following strengthened version  of  de la Pe\~{n}a's inequality (\ref{fin2}):
for all $  x\geq 0 ,$
\begin{eqnarray}
  \ \ \ \ \ \ \ \ \mathbf{P}(S_n>x )  \, \leq \, \left( \frac{}{}1- \Phi\left( \widehat{x} \right)\right)
  \Big[\frac{}{} 1+ C \Big(1+ \widehat{x} \Big)  \Big(  \lambda ^2 \epsilon +  \lambda  \delta^2 + \epsilon \left| \log  \epsilon
 \right| +   \delta  \Big) \Big],\label{iepx1}
\end{eqnarray}
with $\lambda \in [0, \  \epsilon^{-1})$, $ \lambda  = x + O\big(x^2\epsilon+ x\delta \big) ,  \ \ x^2\epsilon+ x\delta \rightarrow 0;$
see Theorem \ref{th2.2} for details.  We will show that bound  (\ref{iepx1})
strengthens de la Pe\~{n}a's inequality (\ref{fin2}) by adding a factor of  type $\frac{1}{1+ x}$.

Of course, condition (A2) is very restrictive.
Without condition (A2), under solely condition (A1), with a method of Bolthausen (1982), we establish the following
nonuniform Berry-Esseen bound under condition  (A1): for all $x \in \mathbf{R},$
\begin{eqnarray}
\left|\frac{}{} \mathbf{P}(S_n \leq x )-  \Phi\left( x\right) \right|
  &\leq&   C \, \bigg( \left(  1+ x^2 \right)\epsilon|\log \epsilon| \exp\left\{-\frac{ \breve{x} ^2}{2}\right\}  \nonumber \\
  & & \ \ \ \ \ \ \ \ + \ \Big(\mathbf{E}|\! \left\langle S\right\rangle_n -1| + \epsilon^2 \Big)^{1/3}  \exp\left\{-\frac{x^2}{6}\right\} \bigg),\ \ \
\label{fgs2}
\end{eqnarray}
where
\begin{eqnarray}\label{tgn35}
\breve{x}  = \frac{2|x| }{1+\sqrt{1+2|x|\epsilon }}.
\end{eqnarray}
This result has an exponential decaying rate in $x$, compablack to the polynomial decaying rate in the nonuniform
Berry-Esseen bounds of Haeusler and Joos (1988b) and  Joos (1991). 

The bounds (\ref{fgs1}) and (\ref{fgs2}) are closely related to the results of Fan et al.\ (2013) and Grama and  Haeusler (2000).   However,  we complete on these results in three aspects.
First, we establish nonuniform Berry-Esseen bounds, which imply the Cram\'{e}r large deviations of    Fan et al.\ (2013) and Grama and  Haeusler (2000) in the normal range $0\leq x =o(\epsilon^{-1/3})$.
Second, we relax condition (A2) of Fan et al.\ (2013), replacing it by boundedness of the moment $\mathbf{E}|\! \left\langle S\right\rangle_n -1|$.
Third, our bounds hold  for all $x \in \mathbf{R},$ compablack with the range $0\leq x =o(\epsilon^{-1})$ established in Fan et al.\ (2013).

The paper is organized as follows. Our  results are stated and discussed in Section \ref{sec2}.  Some applications
 to linear regressions  and    self-normalized large deviations are presented in Section \ref{sec3}. The proofs of the results are deferblack to Sections \ref{sec4} and \ref{sec5}.

\section{Main Results}\label{sec2}
Assume that we are given a sequence of martingale differences $(\xi _i,\mathcal{F}_i)_{i=0,...,n}, $ defined on some
 probability space $(\Omega ,\mathcal{F},\mathbf{P})$,  where $\xi
_0=0 $ and $\{\emptyset, \Omega\}=\mathcal{F}_0\subseteq ...\subseteq \mathcal{F}_n\subseteq
\mathcal{F}$ are increasing $\sigma$-fields.
\textcolor{black}{Consider the martingale $S=(S_k,\mathcal{F}_k)_{k=0,\dots,n}$, where}
\begin{equation}\label{matingal}
S_{0}=0,\ \ \ \ \ \ \ S_k=\sum_{i=1}^k\xi _i,\quad k=1,...,n.
\end{equation}
\textcolor{black}{Let $\left\langle S\right\rangle $ be its pblackictable quadratic variation:}
\begin{equation}\label{quad}
\left\langle S\right\rangle _0=0,\ \ \ \ \ \ \ \ \left\langle S\right\rangle _k=\sum_{i=1}^k\mathbf{E}[\xi _i^2|\mathcal{F}
_{i-1}],\quad k=1,\dots,n.
\end{equation}
Our  main result is the following nonuniform Berry-Esseen bound for martingales.
\begin{theorem}\label{th2.1}
Assume conditions (A1) and (A2). Then, for all $x \in \mathbf{R}$,
\begin{eqnarray}\label{sdfgs}
&& \left|\frac{}{} \mathbf{P}(S_n \leq x )-  \Phi\left( x\right) \right| \, \leq \, C\,\Big(  1+ x^2 \Big)\Big(  \epsilon|\log \epsilon| +\frac{\delta }{1 +|x| } \Big) \exp\left\{-\frac{\widehat{x}^2}{2}\right\} ,
\end{eqnarray}
\textcolor{black}{where
\begin{eqnarray}\label{fbxr}
\widehat{x}   = \frac{2|x|/\sqrt{1+\delta^2}}{1+\sqrt{1+2|x|\epsilon/(1+\delta^2)}}.
\end{eqnarray}}
\end{theorem}

Let us give some  comments on the main result.
\begin{enumerate}
\item When  $(\xi _i)_{i=1,...,n}$
is a sequence of independent random variables, it is possible to improve the factor $  \epsilon|\log \epsilon| +\frac{\delta }{1 +|x| } $ in (\ref{sdfgs})  to  $\epsilon.$    However, this task is beyond the scope of this paper.
We refer to \cite{FGL14} for related bounds.

\item The exponential factor  $\exp\{- \widehat{x}^2/2 \}$ in (\ref{sdfgs}) has the same exponential decay rate as de la Pe\~{n}a's bound (\ref{fin2}) for all $x.$  For moderate $x$, this  exponential factor  has the exponentially decaying rate $\exp\{-x^2/2\}$ as $\max\{\epsilon, \delta\}\rightarrow 0$.
When the martingale differences are bounded, Joos (1991)  also established two nonuniform Berry-Esseen bounds for martingales with exponential decay rates. However, the  exponential decay rate  in Joos (1991) is much slower than that  of  (\ref{sdfgs}).

\item Inequality (\ref{sdfgs})  implies the following Cram\'{e}r large deviation expansion in the normal range: for all $0 \leq x  \leq \min\{\epsilon^{-1/3}, \delta^{-1} \},$
\begin{equation}\label{crd12}
\frac{\mathbf{P}(S_n > x)}{ 1- \Phi \left( x\right)}=  1 + \theta C\,\Big( 1+ x^3 \Big)\Big(  \epsilon|\log \epsilon| +\frac{\delta }{1 +x   } \Big),
\end{equation}
where $|\theta|\leq 1$.    Indeed, since $
0\leq x-\widehat{x}=O\Big(x^2\epsilon+ |x|\delta \Big), \   x^2\epsilon+ |x|\delta \rightarrow 0,$
and
\begin{eqnarray}
\ \ \ \ \ \  \frac{1}{\sqrt{2\pi}(1+x)}   \leq \Big( 1-\Phi \left( x\right) \Big) \exp\left\{ \frac{x^2}{2}\right\} \leq \frac{1}{\sqrt{\pi}(1+x)} \label{mc}
\end{eqnarray}
for all $x\geq0,$
we easily obtain (\ref{crd12}) from (\ref{sdfgs}). Some earlier results of type (\ref{crd12})  have been established  by Bose (1986a, 1986b), Ra\v{c}kauskas (1990, 1995)  and Grama and Haeusler (2000) for martingales with bounded differences; see also    Fan et al.\ (2013) under conditions (A1) and (A2). Notice that the factor $1+ x^3$ in  (\ref{crd12})
is the best possible. Thus, the factor $1+ x^2$ in (\ref{sdfgs}) also cannot be improved to a smaller one.

\item When $|\xi_{i}|\leq \epsilon$ and  condition  (A2) holds,  Bolthausen (1982)  proved that
\begin{equation}\label{f31}
D(S_n)  := \sup_{x \in \mathbf{R}} \Big|\mathbf{P}(S_n\leq x) - \Phi(x) \Big| \, \leq  \,  C  \,\Big(\epsilon^3 n  \log n  +\delta \Big).
\end{equation}
As pointed out by Bolthausen (1982), the  convergence rate in (\ref{f31}) is sharp in the sense that there exists a sequence of  bounded martingale differences $|\xi_i|\leq C/\sqrt{n}$ satisfying  $\langle S\rangle_{n}=1$ a.s.\ such that
\begin{equation}
\limsup_{n\rightarrow \infty} \sqrt{n}\  (\log n)^{-1} D(S_n)  \, >  \, 0.
\label{optim001}
\end{equation}
The factor $\log n$ in the convergence rate is the major difference between the Berry-Esseen bounds for
 martingale difference arrays (under suitable conditions) and for i.i.d.\ sequences, where the Berry-Esseen bounds are of order $ 1/\sqrt{n}$.
It is known that under certain stronger condition (for instance $\mathbf{E}[\xi _i^2|\mathcal{F}
_{i-1}]=1/n,$  $\mathbf{E}[\xi _i^3|\mathcal{F}
_{i-1}]=C_1/n^{3/2}$ and $\mathbf{E}[|\xi _i|^{3+\delta}|\mathcal{F}
_{i-1}]\leq C_2/n^{(3+\delta)/2}, \delta> 0,$ a.s.\ for all $1\leq i \leq n$),  an uniform Berry-Esseen bound of
order $1/\sqrt{n}$ for  martingale difference arrays is possible; we refer to Renz (1996) (see also Bolthausen (1982)).

\ \ \ \ It is easy to see that inequality (\ref{sdfgs}) implies the following Berry-Esseen bound
\begin{equation}\label{brmti}
D(S_n)    \, \leq \,  C\,\Big( \epsilon \left| \log  \epsilon \right|+\delta \Big).
\end{equation}
For martingales with bounded differences (\ref{brmti}) has been established earlier in Grama (1987a,b, 1988).
Note that (\ref{brmti}) implies Bolthausen's inequality (\ref{f31}) under the less restrictive condition (A1).
Indeed, by condition (A2), we have $3/4 \leq \langle X \rangle_n \leq n \epsilon ^2$ a.s. and then $\epsilon \geq \sqrt{3/(4n)}$. For $\epsilon \leq 1/2$, it is easy to see that $\epsilon^3 n   \log n  \geq 3\,\epsilon|\log\epsilon | /4$.
Thus (\ref{brmti}) implies (\ref{f31}).
Note that the bound in (\ref{f31}) may converge to infinity while that in (\ref{brmti}) converges to $0$ as $\epsilon, \delta \to 0$ and $n \to \infty.$
 For instance, if $\epsilon$ is of the order $n^{-1/3}$ and $\delta=o(1)$ as $n \to \infty$, then it is obvious that $\epsilon \left| \log  \epsilon \right|=O(n^{-1/3} \log n  )$ while $\epsilon^3 n  \log n  \geq \log n.$ Thus inequality (\ref{brmti}) strengthens  Bolthausen's inequality (\ref{f31}).

\item When  condition (A2) fails,  the optimal Berry-Esseen bounds for martingales have been obtained by several authors; we refer to  Bolthausen (1982), Haeusler (1988a),  Grama (1988, 1993) and Mourrat (2013). In these papers, the authors assume that the random variable $\langle S\rangle_{n}-1$ has finite moments.
\end{enumerate}

Of course, condition (A2) in our theorem may be very restrictive.
Using  the  method from Bolthausen (1982), we  deduce from (2.3) the following nonuniform Berry-Esseen bound where the condition (A2) is relaxed.

\begin{corollary}\label{co2.1}
Assume condition  (A1). Then, for all $x \in \mathbf{R},$
\begin{eqnarray} \label{thonineqs}
\ \ \ \ \ \Big| \mathbf{P}(S_n \leq x )-  \Phi\left( x\right) \Big|  \!\!  & \leq & \!\!  C \,  \bigg(\! \left(  1+ x^2 \right)\epsilon \, |\log \epsilon| \exp\left\{-\frac{ \breve{x} ^2}{2}\right\} \\
& &\ \ \ \ \  + \ \Big( \mathbf{E} |\! \left\langle S\right\rangle_n -1| + \epsilon^2 \Big)^{1/3}  \exp\left\{-\frac{x^2}{6}\right\} \bigg),\nonumber
\end{eqnarray}
where $ \breve{x} $ is defined by (\ref{tgn35}).
In particular,
\begin{eqnarray} \label{itmab}
\sup_{x \in \mathbf{R}} \Big|\mathbf{P}(S_n\leq x) - \Phi(x) \Big| \, \leq \, C \,  \bigg( \Big(\mathbf{E} |\! \left\langle S\right\rangle_n -1| \Big)^{1/3} + \epsilon^{2/3} \bigg) .
\end{eqnarray}
\end{corollary}

For earlier results for martingales with bounded  differences, we refer to Joos (1991),
where the slower rate $\exp\left\{-\frac{|x|}{5 \log(1+|x|)}\right\}$ has been obtained.
\textcolor{black}{If in the hypothesis (A1) the conditional expectation is replaced by the non-conditional expectation,
then the convergence rate is less sharp than that in \eqref{thonineqs},
as shown by Theorem 3.2 of Lesigne and Voln\'y (2001).}

Inspecting the proof of Theorem 1.5 of  Mourrat (2013),
we see that using the Burkholder inequality instead of Chebyshev  inequality (\ref{chineqd}),
in the proof of Corollary \ref{co2.1},
the bound (\ref{itmab}) can be generalized  to
\begin{eqnarray} \label{itmhklm}
\sup_{x \in \mathbf{R}} \Big|\mathbf{P}(S_n\leq x) - \Phi(x) \Big|\, \leq \, C_p \,   \bigg( \Big(\mathbf{E}[| \left\langle S\right\rangle_n -1|^p] \Big)^{1/(2p+1) } + \epsilon^{2p/(2p+1)}   \bigg)
\end{eqnarray}
for any  $p\geq 1$, where $C_p$ depends only on $p$. Moreover,
 Mourrat (2013) (see also Grama (1988) and Haeusler   (1988a)) showed that when $|\xi_i|\leq C/\sqrt{n},$ a.s.,  the term $\big(\mathbf{E}[| \left\langle S\right\rangle_n -1|^p] \big)^{1/(2p+1) }$ in the bound (\ref{itmhklm}) is sharp in the sense of that (\ref{optim001}) holds true.
 With $p=1$ this implies that the term $\big(\mathbf{E}| \left\langle S\right\rangle_n -1| \big)^{1/3 }$ in (\ref{itmab}) is also sharp.

To prove Theorem \ref{th2.1}, we establish the following strengthened version of de la Pe\~{n}a's  inequality (\ref{fin2}).

\begin{theorem}\label{th2.2}
Assume conditions (A1) and (A2). Then,
for all $  x\geq 0 ,$
\begin{eqnarray}\label{iepx}
&&  \mathbf{P}(S_n>x )   \leq  \Big( \frac{}{}1- \Phi\left( \widehat{x} \right)\Big) \Big[\frac{}{} 1+ C \,\Big(1+\widehat{x} \Big)  \Big( \lambda ^2 \epsilon +  \lambda  \delta^2 + \epsilon \left| \log  \epsilon
 \right| +   \delta \Big )  \Big],
\end{eqnarray}
where $\widehat{x}$ is defined by (\ref{fbxr})
and
\begin{equation}\label{lambda}
 \lambda  = \frac{2x/(1+ \delta^2)}{ 1+2x\epsilon/(1+ \delta^2)+\sqrt{1+2x\epsilon/(1+ \delta^2)} } \in [0,   \epsilon^{-1}).
\end{equation}
In particular, for all $0 \leq x =o\Big(( \sqrt[3]{\epsilon}   + \delta )^{-1}  \Big)$
\begin{eqnarray}\label{f534s}
  \mathbf{P}(S_n>x ) \,\leq\,  \left(\frac{}{} 1- \Phi\left(\widehat{x}\right)\right) \left[1+ o(1) \right],
  \quad \max\{\epsilon, \delta\}\rightarrow 0.
\end{eqnarray}
\end{theorem}

Since the martingale differences $(-\xi _i,\mathcal{F}_i)_{i=1,...,n}$ also satisfy conditions (A1) and (A2), the bound  (\ref{iepx})   can be also applied to obtain upper bounds of $\mathbf{P}(S_n< - x), x>0.$
\textcolor{black}{Our bound  (\ref{iepx}) decays exponentially to zero as $x\rightarrow \infty$ and also recovers closely the shape of the standard normal tail $1-\Phi(x)$ as $\max\{\epsilon, \delta\}\rightarrow 0$.}

 Using the two sides bound (\ref{mc}), it is easy to see that (\ref{iepx}) implies the following inequality: for all $ x \geq 0 ,$
\begin{eqnarray}
  \mathbf{P}(S_n>x )  \leq   F(x)  \exp\left\{ - \frac{ \widehat{x}^2 }{2} \right\} \label{f9},
\end{eqnarray}
where $F(x)=C \left(\frac{1}{1+ \widehat{x}} +  \lambda ^2 \epsilon +  \lambda  \delta^2 + \epsilon \left| \log  \epsilon
 \right| +   \delta   \right)\!.$
Note that when $\max\{\epsilon, \delta\}\rightarrow 0,$  we have $ F(x) \rightarrow \frac{C}{1+x}.$ Inequality (\ref{f534s}) strengthens de la Pe\~{n}a's inequality (\ref{fin2}) by adding a factor of order $\frac{1}{1+ x}$ as $\max\{\epsilon, \delta\}\rightarrow 0$. Thus (\ref{iepx}) is a strengthened version of  de la Pe\~{n}a's inequality (\ref{fin2}).

Our result (\ref{iepx})  can be compablack with the classical Cram\'{e}r  large deviation results in the i.i.d.\ case (see Cram\'{e}r (1938)). With respect to Cram\'{e}r's results, the advantage of (\ref{iepx}) is that it is valid for all $x\geq0$ instead of only for all $0\leq x =(\sqrt{n}), n\rightarrow \infty.$

\section{Applications in Statistics}\label{sec3}

\subsection{Linear regression}
The linear regression model is given by
\begin{equation}\label{ine29}
X_{k}=\theta \phi_k + \varepsilon_{k}, \quad  1 \leq k \leq n,
\end{equation}
where $(X_k), (\phi_k)$ and $(\varepsilon_{k})$ are, respectively,  the responce variable, the covariate and the noise. We assume that $(\phi_k)$ is a sequence of independent  random variables. We also assume that $(\varepsilon_k)$ is a sequence of martingale differences with respect to its natural filtration $\mathcal{F}_k=\sigma\{\varepsilon_i, 1\leq i  \leq k \}$  with conditional variances satisfying $\mathbf{E}[\varepsilon_k^2| \mathcal{F}_{k-1}]=\sigma^2>0 $ a.s.   Moreover, we suppose that the sequences $(\phi_k)$ and  $(\varepsilon_k)$ are independent. Our interest is to estimate the unknown parameter $\theta$, based on the random variables $(X_k)$ and $(\phi_k)$.   The well-known least squares estimator $\theta_n$ is given by
\begin{equation}\label{ine30}
\theta_n = \frac{\sum_{k=1}^n \phi_{k} X_k}{\sum_{k=1}^n \phi_{k}^2}.
\end{equation}
To measure the accuracy of the convergence $\theta_n \to \theta$ as $n\rightarrow \infty$, many tight exponential inequalities on tail probability of $\theta_n -\theta$ have been established in Bercu and Touati (2008).
Here, we give an  estimation on the rate of convergence of the distribution of
$(\theta_n -\theta)\sqrt{\sum_{k=1}^n \phi_{k}^2}$ to the normal one.
\begin{theorem}\label{th3.1} \textcolor{black}{Assume that there exist  two positive numbers $\epsilon_1$ and $\epsilon_2$ satifying
$\epsilon :=  \epsilon_1 \epsilon_2/ \sigma \leq \frac12$
and such that}
\begin{equation}
\frac{|\phi_{k}|}{\sqrt{\sum_{j=1}^n \phi_{j}^2}} \leq \epsilon_1  \  \ \    \textrm{ a.s.\ for all}\ 1\leq k\leq n
\label{ttt001}
\end{equation}
and
\begin{equation}
\Big|\mathbf{E}[ \varepsilon_{k}^{l} | \mathcal{F}_{k-1} ]  \Big| \leq \frac12\, l!\, \epsilon_2^{l-2} \sigma^2 \  \   \ \ \textrm{a.s.\ for all}\ l\geq 2\ \textrm{and all}\  1\leq k\leq n.
\label{ttt002}
\end{equation}
Then, for all $x \in \mathbf{R}$,
\begin{eqnarray}\label{th5ineq}
 \Big|\mathbf{P}\Big( (\theta_n -\theta)\sqrt{ \Sigma_{k=1}^n \phi_{k}^2} \leq x \sigma  \Big) - \Phi(x)  \Big| \, \leq \, C\, \left(\frac{}{}\! 1+ x^2 \right) \epsilon  \left|\log \epsilon  \right|  \exp\left\{-\frac{ \breve{x} ^2}{2}\right\} ,
\end{eqnarray}
where $ \breve{x} $ is defined by (\ref{tgn35}).
Moreover,   the following Berry-Esseen bound holds
\begin{eqnarray}
  \Big|\mathbf{P}\Big( (\theta_n -\theta)\sqrt{ \Sigma _{k=1}^n \phi_{k}^2} \leq x \sigma  \Big) - \Phi(x)  \Big| \, \leq \, C\ \epsilon \left|\log \epsilon  \right|.
\end{eqnarray}
For all $ 0\leq x  =O\big(\epsilon^{-1/3}  \big)$ as $\epsilon  \rightarrow 0,$ the following Cram\'{e}r large deviation result holds
\begin{eqnarray}\label{th5ineq}
  \frac{\mathbf{P}\left( (\theta_n -\theta)\sqrt{ \Sigma _{k=1}^n \phi_{k}^2} \geq x \sigma  \right)}{1-  \Phi(x) } =1+   \vartheta \, C \,\left(\frac{}{}\! 1+ x^3 \right) \epsilon \left|\log \epsilon  \right| ,
\end{eqnarray}
where $|\vartheta|\leq 1.$
\end{theorem}

In the real-world applications, for instance considering the impact of the footprint size $\phi_{k}$ on the height $X_k$, it is plausible that
$a\leq \phi_{k} \leq b $ a.s.\ for two positive constants $a$ and $b$ and that $ \epsilon_2$ is a constant. In particular, if $(\varepsilon_{k})_{k=1,...,n}$ is a sequence of independent
random variables satisfying the Bernstein  condition
\[
  \Big|\mathbf{E} [ \varepsilon_{k}^l ] \Big|  \,  \leq \, \frac12  l! C^{l-2}  \mathbf{E} [\varepsilon_{k}^2 ]  \ \ \textrm{for all}\ l\geq 2\  \textrm{and all}\ 1\leq k\leq n,
\]
then the conditions (\ref{ttt001}) and  (\ref{ttt002}) of Theorem \ref{th3.1} are  satisfied with $\epsilon_1 = \frac{b}{a \sqrt{n  } }$ and $\epsilon=O(\frac1{\sqrt{n}})$ as $n \rightarrow \infty.$

\subsection{Self-normalized deviations}

The self-normalized deviations have attracted lots of attentions due to the useful application to Student's $t$-statistic; we refer to  Shao (1997, 1999) and Jing, Shao and Wang (2003).
 The first exponential nonuniform Berry-Esseen bound for the self-normalized mean for symmetric
random variables $(\xi_i)_{i=1,...,n}$ has been established by Wang and Jing (1999) (cf.\ Theorem 2.1 therein): if $\mathbf{E}[|\xi_i|^3] < \infty$ for all $i \in [1, n],$ then
\begin{eqnarray}
&&  \bigg|\frac{}{} \mathbf{P}\bigg( \frac{S_n}{\sqrt{[S]_n}} \leq x \bigg)-  \Phi\left( x\right) \bigg| \ \ \ \ \ \ \ \   \ \ \ \    \ \ \ \   \ \ \ \ \  \  \nonumber \\
&& \leq\left\{ \begin{array}{ll}
C\bigg( L_{3n}   \Big( 1+ x^2  \Big)+ \sum_{i=1}^n\mathbf{P}\Big(|\xi_i|\geq B_n(6|x|)^{-1} \Big) \bigg)\exp\bigg\{-\frac{ x^2}{2}\bigg\} ,\ \\
   \ \ \ \ \ \ \ \  \ \   \ \ \ \  \ \   \ \ \ \    \ \ \ \  \ \   \ \ \ \  \ \  \ \ \  \ \ \ \ \  \ \   \ \ \ \ \ \ \ \  \ \  \  \ \  \  \ \ \   \textrm{if\ \   $|x| \leq \Big( 5 L_{3n}^{1/3} \Big)^{-1}$,  }
\\
\bigg(  1+  \frac{1}{\sqrt{2\pi} |x|} \bigg)\exp\bigg\{-\frac{ x^2}{2}\bigg\},   \ \ \ \   \ \ \ \  \ \ \ \   \ \ \ \    \ \ \ \  \ \ \ \     \textrm{if\ \  $|x| > \Big( 5 L_{3n}^{1/3} \Big)^{-1}$ },
\end{array} \right. \ \ \ \ \  \ \ \ \label{wjineq}
\end{eqnarray}
where $B_n^2=\mathbf{E}[S_n^2],$ $L_{3n}= B_n^{-3}\sum_{i=1}^n \mathbf{E}[|\xi_i|^3]$ and $[S]_n=\sum_{i=1}^n\xi_i^2 $ is  the  square bracket  of $S_n$.
Here, by convention, we assume $\frac00=0.$
In the following theorem, we give a  result  similar to  inequality (\ref{wjineq}) of Wang and Jing (1999) via a new method based on martingale theory.  In particular, our result
 does not depend on the moments of random variables.

\begin{theorem}\label{th3.3} Let $(\xi_i)_{i=1,...,n}$ be a sequence of non-degenerate, independent and symmetric random variables. If there exists a number  $\epsilon \in  (0, \frac12]$, possibly depending on $n$, such that
\begin{eqnarray}\label{fhsnnl}
  \frac{|\xi_i| }{ \sqrt{[S]_n}  }  \leq \epsilon \ \ \ \ \ \ \textrm{a.s.\ for all} \ \  i \in [1, n],
\end{eqnarray}
then, for all $x \in \mathbf{R}$,
\begin{eqnarray} \label{fgs}
  \ \ \  \Big| \mathbf{P}\Big( \frac{S_n}{\sqrt{[S]_n}} \leq x \Big)-  \Phi\left( x\right) \Big|  \,\leq \, C\, \epsilon \, |\log \epsilon|  \left(\frac{}{}\! 1+ x^2 \right) \exp\left\{-\frac{ x ^2}{2}\right\} ,
\end{eqnarray}
and, for all $x  \geq 0$,
\begin{eqnarray}\label{fgs01}
  \frac{\mathbf{P}\Big( \frac{S_n}{\sqrt{[S]_n}} > x  \Big)}{1-  \Phi(x) } \ = \  1+   \vartheta \, C \,\left(\frac{}{}\! 1+ x^3 \right)  \, \epsilon \, |\log \epsilon| ,
\end{eqnarray}
where $|\vartheta| \leq 1.$
\end{theorem}

We continue with some comments and remarks on the obtained results.
\begin{enumerate}

\item Condition (\ref{fhsnnl}) in Theorem \ref{th3.3} is  satisfied when $\xi_i$ are all bounded. For instance,  if $a\leq |\xi_i | \leq b $ a.s.\ for two positive constants $a$ and $b$, then it is obvious that the condition of Theorem \ref{th3.3} is  satisfied with $\epsilon = \frac{b}{a \sqrt{n  } }=O(\frac1{\sqrt{n}})$ as $n \rightarrow \infty.$

\item We conjecture that when $(\xi_i)_{1\leq i \leq n}$ are symmetric and i.i.d.\  non-degenerate random variables with moments of order $3$, the following results (\textcolor{black}{similar to that of Fan, Grama  and Liu (2015a)}) hold: for all $0 \leq x =o(\sqrt{n}),$
\begin{eqnarray}\label{fgscgh01}
\ \ \ \ \ \ \ \ \   \mathbf{P}\Big( \frac{S_n}{\sqrt{[S]_n}} > x \Big) \, =\,
\inf_{\lambda\geq 0} \mathbf{E} \Big[ e^{\lambda (\frac{S_n}{\sqrt{[S]_n}}- x) } \Big]
\Big(1+ O\Big(\frac{1+x}{\sqrt{n}}\Big)  \Big),
\end{eqnarray}
and, for all $ 0 \leq x =o( n^{1/4} ) $, with $|\vartheta| \leq 1,$
\begin{eqnarray}\label{fgscgh}
  \frac{\mathbf{P}\Big( \frac{S_n}{\sqrt{[S]_n}} > x  \Big)}{1-  \Phi(x) } \ = \  1+   \vartheta \, C \, \frac{1+ x^3 }{\sqrt{n}}.
\end{eqnarray}

\item For more Cram\'{e}r-type large deviation results on self-normalized sequences, we refer to  Shao (1997) and Jing, Shao and Wang (2003).  In particular, without assuming that $(\xi _i)_{i=1,...,n}$ are symmetric, equalities of type (\ref{fgscgh}) in the range $0 \leq x =o( n^{1/6} )$ have been established therein. This range of $x$ is the best possible for non-symmetric random variables $(\xi _i)_{i=1,...,n}.$ 

\end{enumerate}

\section{Proofs of Theorems \ref{th2.1} - \ref{th2.2}}\label{sec4}
In this section we prove the main results of the paper.
We start with some auxiliary statements, then we prove Theorem  \ref{th2.2}   which in turn will be used in the proof of Theorem  \ref{th2.1}.

\textcolor{black}{In the sequel, for simplicity, the equalities and inequalities involving random variables will be understood in the a.s. sense without mentioning this.
}

\subsection{Auxiliary results}
Let $(\xi _i,\mathcal{F}_i)_{i=0,...,n}$ be a sequence of martingale differences
satisfying condition (A1) and $S=(S_k,\mathcal{F}_k)_{k=0,...,n}$
be the corresponding martingale defined by (\ref{matingal}). For any real number $\lambda$ with $|\lambda| < \epsilon^{-1} ,$
define the \emph{exponential multiplicative martingale} $Z(\lambda
)=(Z_k(\lambda ),\mathcal{F}_k)_{k=0,...,n},$ where
\[
Z_0(\lambda )=1,\quad \quad \quad Z_k(\lambda )=\prod_{i=1}^k\frac{e^{\lambda \xi _i}}{\mathbf{E}[e^{\lambda \xi _i}|
\mathcal{F}_{i-1}]},\quad k=1,...,n.  \label{C-1}
\]
For each $k=1,...,n,$ the random variable $Z_k(\lambda
) $ defines a probability density on $(\Omega ,\mathcal{F},\mathbf{P}).$ This allows us to introduce, for $|\lambda|
 <\epsilon^{-1},$ the \emph{conjugate probability measure} $\mathbf{P}_\lambda $ on $(\Omega ,%
\mathcal{F})$ defined by
\begin{equation}
d\mathbf{P}_\lambda =Z_n(\lambda )d\mathbf{P}.  \label{f21}
\end{equation}
Denote by $\mathbf{E}_{\lambda}$ the expectation with respect to $\mathbf{P}_{\lambda}$.
For all $i=1,\dots,n$, let
\[
\eta_i(\lambda)=\xi_i - b_i(\lambda)\ \ \ \ \ \ \ \ \  \textrm{and} \ \ \ \  \ \ \  \ \  b_i(\lambda)=\mathbf{E}_{\lambda}[\xi_i |\mathcal{F}_{i-1}].  \]
We thus obtain the following decomposition:
\begin{equation}
X_k=Y_k(\lambda )+B_k(\lambda ),\quad\quad\quad k=1,...,n, \label{xb}
\end{equation}
where $Y(\lambda )=(Y_k(\lambda ),\mathcal{F}_k)_{k=1,...,n}$ is the \emph{%
conjugate martingale} defined as
\begin{equation}\label{f23}
Y_k(\lambda )=\sum_{i=1}^k\eta _i(\lambda ),\quad\quad\quad k=1,...,n,
\end{equation}
and
$B(\lambda )=(B_k(\lambda ),\mathcal{F}_k)_{k=1,...,n}$ is the \emph{%
drift process} defined as
\[
B_k(\lambda )=\sum_{i=1}^kb_i(\lambda ),\quad\quad\quad k=1,...,n.
\]
In the proofs of theorems, we shall  make use of the following bounds of $B_n(\lambda )$.
\begin{lemma}
\label{lemma1} Assume conditions (A1) and (A2). Then, for all $0 \leq \lambda < \epsilon^{-1} ,$
\begin{eqnarray*}
\lambda -\lambda\delta^2 - C \lambda^2\epsilon    \ \leq \  B_n(\lambda ) \ \leq \  \frac{\lambda-0.5\lambda^2\epsilon}{ (1-\lambda\epsilon)^2} \left(1+\delta^2 \right). 
\end{eqnarray*}
\end{lemma}
\begin{proof}
 Since $\mathbf{E}[\xi _i|\mathcal{F}_{i-1}]=0$ and $\lambda\geq 0$,  it follows that
 \[
\mathbf{E}[\xi_{i} e^{\lambda\xi_{i}} |\mathcal{F}_{i-1}]=\mathbf{E}[\xi_{i}(e^{\lambda\xi_{i}}-1)|\mathcal{F}_{i-1}]\geq 0
\]
and, by Jensen's inequality,
$\mathbf{E}[e^{\lambda \xi _i}|\mathcal{F}_{i-1}]\geq 1$.
Using Taylor's expansion of $e^x$, we get
\begin{eqnarray}\label{bn}
B_n(\lambda ) & =& \sum_{i=1}^{n} \frac{\mathbf{E}[\xi_{i} e^{\lambda \xi_{i}} | \mathcal{F}_{i-1}]}{\mathbf{E}[ e^{\lambda \xi_{i}} | \mathcal{F}_{i-1}]}  \nonumber\\
& \leq & \sum_{i=1}^{n}\mathbf{E}[\xi_{i} e^{\lambda \xi_{i}} | \mathcal{F}_{i-1}]\nonumber\\
 & = &  \lambda\langle S\rangle_{n}+ \sum_{i=1}^{n}\sum_{k=2}^{+\infty}\mathbf{E}\bigg[\frac{\xi_{i}(\lambda\xi_{i})^{k}}{k !} \bigg| \mathcal{F}_{i-1} \bigg]   .
\end{eqnarray}
By condition (A1), for all $0 \leq \lambda < \epsilon^{-1} ,$ we deduce
\begin{eqnarray}\label{sum29}
  \sum_{i=1}^{n}\sum_{k=2}^{+\infty}\bigg| \mathbf{E}\bigg[\frac{\xi_{i}(\lambda\xi_{i})^{k}}{k !} \bigg| \mathcal{F}_{i-1} \bigg] \bigg|
  & = & \sum_{i=1}^{n}\sum_{k=2}^{+\infty}|\mathbf{E}[\xi_{i}^{k+1}| \mathcal{F}_{i-1}]| \frac{\lambda^k}{k !} \ \nonumber\\
  & \leq & \frac12\, \lambda^2 \epsilon \,  \langle S \rangle_{n} \sum_{k=2}^{+\infty}(k+1)(\lambda\epsilon)^{k-2} \nonumber \\
& = & \frac{(3-2\lambda\epsilon)}{2(1-\lambda\epsilon)^2}\, \lambda^2 \epsilon \, \langle S \rangle_{n}  .
\end{eqnarray}
Inequalities (\ref{sum29}) and (\ref{bn}) imply that, for all $0 \leq \lambda < \epsilon^{-1} ,$
\begin{eqnarray*}
 B_n(\lambda )  \leq  \lambda\langle S \rangle_{n} +  \frac{(3-2\lambda\epsilon)}{2(1-\lambda\epsilon)^2}\, \lambda^2 \epsilon  \langle S \rangle_{n}
  =   \frac{\lambda-0.5\lambda^2\epsilon}{ (1-\lambda\epsilon)^2} \langle S \rangle_{n}.
\end{eqnarray*}
This bound together with condition (A2) gives  the upper bound of $B_{n}(\lambda)$. The lower bound of $B_{n}(\lambda)$ can be found in Lemma 4.2 of \cite{FGL13}.
\end{proof}

Consider the pblackictable process $\Psi (\lambda )=(\Psi
_k(\lambda ),\mathcal{F}_k)_{k=0,...,n},$ where
\begin{equation}
\Psi _k(\lambda )=\sum_{i=1}^k\log \mathbf{E}[e^{\lambda \xi _i}|\mathcal{F}_{i-1}].
\label{C-3}
\end{equation}
We have the following elementary bound.
\begin{lemma}
\label{lemma2} Assume conditions (A1) and (A2).  Then, for all $0 \leq \lambda < \epsilon^{-1} ,$
\begin{eqnarray} \label{ghnle2}
\Psi _n(\lambda )  \leq    \frac{\lambda^2(1+\delta^2)}{2 (1-\lambda\epsilon)}.
\end{eqnarray}
\end{lemma}
\begin{proof}  Since $\log (1+t) \leq t$ for all $t\geq 0$,  we have, for all $0 \leq \lambda < \epsilon^{-1} ,$
\begin{eqnarray*}
\Psi _n(\lambda )  = \sum_{i=1}^{n} \log \Big(1+ \mathbf{E}[e^{\lambda \xi _i}|\mathcal{F}_{i-1}]-1  \Big)
 \leq \sum_{i=1}^{n}  \Big(  \mathbf{E}[e^{\lambda \xi _i}|\mathcal{F}_{i-1}]-1  \Big).
\end{eqnarray*}
By condition (A1), it is easy to see that, for all $0 \leq \lambda < \epsilon^{-1} ,$
\begin{eqnarray*}
 \mathbf{E}[e^{\lambda \xi _i}|\mathcal{F}_{i-1}]-1 &=& \sum_{k=2}^{+\infty}\frac{\lambda^{k}}{k !} \mathbf{E}[\xi_{i}^{k} |\mathcal{F}_{i-1}]  \nonumber\\
&\leq & \frac{\lambda^2}{2}\mathbf{E}[\xi_{i}^{2} |\mathcal{F}_{i-1}]\sum_{k=2}^{\infty}(\lambda\epsilon)^{k-2}  \nonumber\\
 &= &\frac{\lambda^2\mathbf{E}[\xi_{i}^{2} |\mathcal{F}_{i-1}]}{2(1-\lambda\epsilon)}.\nonumber
\end{eqnarray*}
Thus, we obtain, for all $0 \leq \lambda < \epsilon^{-1} ,$
\begin{eqnarray}
\Psi _n(\lambda )   \leq   \frac{\lambda^2\langle S\rangle_n}{2 (1-\lambda\epsilon)}.\nonumber
\end{eqnarray}
This inequality together with condition (A2) gives inequality (\ref{ghnle2}).
\end{proof}

\subsection{Proof of Theorem \ref{th2.2}}
For all $ 0\leq \lambda <\epsilon^{-1}$ and $ x \geq0$,  by the definition of the conjugate measure (\ref{f21}), we have the following representation:
\begin{eqnarray*}
\ \ \ \mathbf{P}(S_n>x ) &= & \mathbf{E}_\lambda \left[ Z_n (\lambda)^{-1}\mathbf{1}_{\{S_n>x\}} \right]\\
&= & \mathbf{E}_\lambda \left[ \exp \left\{
-\lambda S_n+\Psi _n(\lambda )\right\}  \mathbf{1}_{\{S_n>x\}}\right] \nonumber\\
&= & \mathbf{E}_\lambda \Big[  \exp \left\{-\lambda x +\Psi _n(\lambda )
 -\lambda ( Y_n(\lambda) + B_{n}(\lambda) - x ) \right\}  \nonumber \\
 && \ \ \ \  \ \ \  \times \mathbf{1}_{\{Y_{n}(\lambda)+B_n(\lambda)-x>0\}}  \Big].\nonumber
\end{eqnarray*}
Setting $ U_n(\lambda)= \lambda (Y_{n}(\lambda)+B_n(\lambda)-x)$, we deduce,  for all $0\leq \lambda < \epsilon^{-1} ,$
\begin{eqnarray}
 \mathbf{P}(S_n>x ) &\leq & \exp \left\{
- \lambda x +\frac{\lambda^2(1+\delta^2)}{2 (1-\lambda\epsilon)} \right\} \mathbf{E}_\lambda \Big[ e^{ - U_n(\lambda)  }   \mathbf{1}_{\{U_n(\lambda)>0\}} \Big] \nonumber\\
&=&  \exp \left\{
- \lambda x +\frac{\lambda^2(1+\delta^2)}{2 (1-\lambda\epsilon)} \right\} \int_0^{\infty}  e^{ -  t  } \mathbf{P}_\lambda(0 < U_n(\lambda) \leq t) dt.\label{f36}
\end{eqnarray}
To optimize the term in the last exponent, let
$\overline{\lambda }=\overline{\lambda }(x) \in [0, \epsilon^{-1})$ be the
unique solution of the equation
\begin{equation}\label{defranda}
-  x + \frac{d}{d \lambda} \Big(\frac{\lambda^2(1+\delta^2)}{2 (1-\lambda\epsilon)} \Big)=0, \ \  \ \mbox{or equivalently} \ \ \  \frac{\lambda-0.5\lambda^2\epsilon}{ (1-\lambda\epsilon)^2} =\frac{x}{1+\delta^2}.
\end{equation}
The definition of $\overline{\lambda }$ and Lemma \ref{lemma1} imply that, for all $  x \geq 0,$
\begin{equation}
  \overline{\lambda }= \frac{2x/(1+ \delta^2)}{ 1+2x\epsilon/(1+ \delta^2)+\sqrt{1+2x\epsilon/(1+ \delta^2)} }. \label{fsf2}
\end{equation}
Taking $\lambda =\overline{\lambda }$  in (\ref{f36}),
we get
\begin{eqnarray} \label{ff32}
\mathbf{P}(S_n>x )
 \leq    \exp \left\{- \frac{\hat{x}^2}{2}  \right\}   \int_0^{\infty}  e^{ -  t  } \mathbf{P}_{\overline{\lambda}}(0 < U_n(\overline{\lambda}) \leq t) dt,
\end{eqnarray}
where
\begin{eqnarray}\label{fineq40}
\hat{x}=\frac{\overline{\lambda}\sqrt{1+\delta^2}}{1-\overline{\lambda}\epsilon}=\frac{2x/\sqrt{1+\delta^2}}{1+\sqrt{1+2x\epsilon/(1+\delta^2)}}.
\end{eqnarray}
To bound the integral term of (\ref{ff32}), we make use of the following lemma, which gives
 a rate of convergence in the central limit theorem for the
process $U( \overline{\lambda}  )$ under the conjugate probability measure $\mathbf{P}_{ \overline{\lambda}  }.$ The proof of this lemma is given in Appendix A.
\begin{lemma}
\label{LEMMA4}
Assume conditions (A1) and (A2). Then, for all $0 \leq \overline{\lambda} < \epsilon^{-1},$
\[
\sup_{u\in \mathbf{R}}\Big| \mathbf{P}_{\overline{\lambda}} (  U_n(\overline{\lambda} )   \leq  \widehat{x}u)-\Phi (  u)\Big| \leq
C\Big( \overline{\lambda}^2 \epsilon + \overline{\lambda} \delta^2 + \epsilon \left| \log  \epsilon
 \right| +   \delta \Big) .
\]
\end{lemma}
Now we return to the proof of Theorem \ref{th2.2}.  From (\ref{ff32}), using Lemma {\ref{LEMMA4}},
it follows that, with $\mathcal{N}(0,1)$ the standard normal random variable,
\begin{eqnarray}
\int_0^{\infty}  e^{ -  t  } \mathbf{P}_{\overline{\lambda}}(0 < U_n(\overline{\lambda}) \leq t) dt
&=&   \int_0^{\infty}  \exp\{ - \widehat{x}  t  \}\ \mathbf{P}_{\overline{\lambda} }
(0 < U_n(\overline{\lambda})\leq \widehat{x}t) \widehat{x} dt \nonumber \\
&\leq&   \int_0^{\infty}  \exp\{ - \widehat{x}   t  \}\ \mathbf{P} _{\overline{\lambda}}
(0 < \mathcal{N}(0,1) \leq t) \widehat{x}dt \nonumber \\
& &  + \, C\, \Big( \overline{\lambda}^2 \epsilon + \overline{\lambda} \delta^2 + \epsilon \left| \log  \epsilon
 \right| +   \delta \Big) \nonumber \\
&\leq&  \exp\left\{ \frac { \widehat{x}^2} { 2}     \right\} \Big(1 - \Phi(\widehat{x}) \Big)  \label{sgbs}\\
& &  + \, C\, \Big( \overline{\lambda}^2 \epsilon + \overline{\lambda} \delta^2 + \epsilon \left| \log  \epsilon
 \right| +   \delta  \Big). \nonumber
\end{eqnarray}
Combining (\ref{ff32}) and (\ref{sgbs}) together,  we have, for all $ x\geq 0 ,$
\begin{eqnarray*}
 \mathbf{P}(S_n>x ) \, \leq \,  1- \Phi\left( \widehat{x}\right)
  +  C\, \exp \left\{- \frac{\widehat{x}^2}{2 }  \right\}  \Big(\overline{\lambda}^2 \epsilon + \overline{\lambda} \delta^2 + \epsilon \left| \log  \epsilon
 \right| +   \delta  \Big)  .
\end{eqnarray*}
Using the two sided bound (\ref{mc}),
we obtain, for all $ x\geq 0 ,$
\begin{eqnarray*}
 \mathbf{P}(S_n>x ) \, \leq \, \Big( 1- \Phi\left(\widehat{x}\right)\Big)  \Big[ 1+ C \,\Big(1+\widehat{x} \Big)  \Big(\overline{\lambda}^2 \epsilon + \overline{\lambda} \delta^2 + \epsilon \left| \log  \epsilon
 \right| +   \delta \Big ) \Big].\label{fineth1}
\end{eqnarray*}
 This completes the proof of Theorem \ref{th2.2}.\hfill\qed

\subsection{Proof of Theorem  \ref{th2.1}}\label{sec3.2}
We firstly prove that (\ref{fgs}) holds for all $x\geq0$.
From (\ref{fineth1}) and the fact $\mathbf{P}(S_n>x)=1- \mathbf{P}(S_n \leq x)$, it follows that, for all $x\geq0$,
\begin{eqnarray}
 \Phi\left( x \right)-   \mathbf{P}(S_n \leq x )  &=&\mathbf{P}(S_n>x )- \Big( 1- \Phi\left( x \right) \Big)\nonumber\\
&=&\Phi(x)- \Phi\left( \widehat{x} \right) + \mathbf{P}(S_n>x )- \Big( 1- \Phi\left( \widehat{x} \right) \Big)  \nonumber\\
   &\leq& \Phi(x)- \Phi\left( \widehat{x} \right) \nonumber\\
    && + C_1\Big( 1- \Phi\left( \widehat{x} \right) \Big) \Big(1+ \widehat{x} \Big)  \Big(   \overline{\lambda}^2 \epsilon + \overline{\lambda} \delta^2  + \epsilon |\log \epsilon |  + \delta  \Big)   \nonumber\\
    &\leq& \frac{1}{\sqrt{2 \pi}} |x-\widehat{x}|\exp\left\{ -\frac{\widehat{x}^2 }{2} \right\} \label{fineq56}\\
    && + C_2 \,  \Big( x^2 \epsilon + x \delta^2  + \epsilon |\log \epsilon |  + \delta  \Big)\exp\left\{ -\frac{\widehat{x}^2 }{2} \right\}, \nonumber
\end{eqnarray}
where the last line follows from $\overline{\lambda} \leq   x$.
From the definitions of $\overline{\lambda}$  and $\widehat{x}$ (cf. (\ref{defranda}) and  (\ref{fineq40})), we deduce that
\begin{eqnarray}
   |x-\widehat{x}| &=& \frac{\overline{\lambda}(1-0.5\overline{\lambda}\epsilon)(1+\delta^2)}{ (1-\overline{\lambda}\epsilon)^2}- \frac{\overline{\lambda} \sqrt{ 1+\delta^2 }}{ 1-\overline{\lambda}\epsilon }  \nonumber\\
    &=& \frac{\overline{\lambda}( 1+ \delta^2 -   \sqrt{ 1+\delta^2 } \,) }{ (1-\overline{\lambda}\epsilon)^2} +\frac{\overline{\lambda}^2 \epsilon \,\big( \sqrt{ 1+\delta^2 } - 0.5(1+\delta^2) \big) }{ (1-\overline{\lambda}\epsilon)^2} \nonumber\\
     &\leq & \frac{\overline{\lambda}\delta^2  }{ (1-\overline{\lambda}\epsilon)^2} +\frac{\overline{\lambda}^2 \epsilon  }{ (1-\overline{\lambda}\epsilon)^2} \nonumber \\
    &\leq& C\, x\, \Big(\delta^2 +  x\epsilon \Big) .\label{fineq57}
\end{eqnarray}
Combining (\ref{fineq56}) and (\ref{fineq57}) together, we have, for all $x \geq 0$,
\begin{eqnarray}
 \ \ \ \ \ \ \ \ \ \   \Phi\left( x \right)  -  \mathbf{P}(S_n \leq x ) \leq
    C \, \Big(1+x^2 \Big) \bigg( \epsilon |\log \epsilon |  + \frac{\delta}{1+x }  \bigg)\exp\left\{ -\frac{\widehat{x}^2 }{2} \right\}.\label{fupbounds}
\end{eqnarray}
To prove the lower bound of $ \Phi\left( x \right) - \mathbf{P}(S_n \leq x ) $, we shall use the following Cram\'{e}r large deviation expansion: for all $0\leq x \leq  \min\{\epsilon^{-1/3}, \delta^{-1}\},$
\begin{equation}
\frac{\mathbf{P}(S_n>x)}{1-\Phi \left( x\right)}= 1 + \theta C\,\bigg(  \Big(1+x \Big)\Big(\epsilon \left| \log  \epsilon
 \right| + \delta \Big)+ x^3\epsilon \bigg),
\end{equation}
where $|\theta|\leq 1$.  This Cram\'{e}r large deviation expansion is a simple consequence of Corollary 2.1 of \cite{FGL13}. Using the equality  $\mathbf{P}(S_n>x)=1- \mathbf{P}(S_n \leq x)$ and the two-sides bound (\ref{mc}), we get, for all $0\leq x \leq  \min\{\epsilon^{-1/3}, \delta^{-1}\},$
\begin{eqnarray}
 \Phi\left( x \right)-  \mathbf{P}(S_n \leq x )   &\geq& - C_1\, \Big(1-\Phi \left( x\right) \Big) \bigg(  \Big(1+x \Big)\Big(\epsilon \left| \log  \epsilon
 \right| + \delta \Big)+ x^3\epsilon \bigg)  \nonumber\\
 &\geq& - C_2\,\Big( \epsilon \left| \log  \epsilon
 \right| + \delta  + x^2\epsilon \Big) \exp\left\{ -\frac{x ^2 }{2} \right\} \nonumber\\
 &\geq& - C_2\, \Big(1+x^2 \Big)\Big( \epsilon \left| \log  \epsilon
 \right| + \frac{\delta}{1+x }   \Big) \exp\left\{ -\frac{ x  ^2 }{2} \right\}. \label{f34sd1}
\end{eqnarray}
For all $ x \geq  \min\{\epsilon^{-1/3}, \delta^{-1}\}$, we have $\frac{1}{1+ \widehat{x}} \leq C\, ( x^2 \epsilon+ x\delta).$ Then, from (\ref{f9}), it is easy to see that, for all $ x \geq  \min\{\epsilon^{-1/3}, \delta^{-1}\},$
\begin{eqnarray}
  \mathbf{P}(S_n>x )  \leq   C \,  \Big(1+x^2 \Big) \Big( \epsilon \left| \log  \epsilon
 \right| + \frac{\delta}{1+x }   \Big)\exp\left\{ -\frac{\widehat{ x }^2 }{2} \right\}. \label{f34sd2}
\end{eqnarray}
Using (\ref{mc}) and the bound $ x \geq  \min\{\epsilon^{-1/3}, \delta^{-1}\}$, the same upper bound is obtained for $1-\Phi(x)$.
Therefore, we have, for all $ x \geq  \min\{\epsilon^{-1/3}, \delta^{-1}\}$,
\begin{eqnarray}\label{ft46}
&& \Phi\left( x \right) -  \mathbf{P}(S_n\leq x )   \, \geq \,  - C \,    \Big(1+x^2 \Big)\Big(  \epsilon |\log \epsilon |  + \frac{\delta}{1+x } \Big)\exp\left\{ -\frac{\widehat{ x }^2 }{2} \right\}.
\end{eqnarray}
Combining  (\ref{fupbounds}), (\ref{f34sd1}) and (\ref{ft46}) together, we obtain, for all $ x \geq 0$,
\begin{eqnarray*}
&& \left|\frac{}{}   \mathbf{P}(S_n\leq x ) - \Phi\left( x \right)\right|  \,  \leq\, C \,\Big(1+x^2 \Big)\Big( \epsilon |\log \epsilon |  + \frac{\delta}{1+x } \Big)\exp\left\{ -\frac{\widehat{ x }^2 }{2} \right\}.
\end{eqnarray*}
Notice that the same argument is applied to $-S_n$. Thus, for all $x<0$,
\begin{eqnarray}
\left|\frac{}{}   \mathbf{P}(S_n\leq x ) - \Phi\left( x \right)\right| &=& \left|\frac{}{}   \mathbf{P}( -S_n\geq -x ) - \Phi\left( x \right)\right| \nonumber\\
&=& \left|\frac{}{}   \mathbf{P}( -S_n\geq -x )-1 - \Big(\Phi\left( x \right)-1 \Big)\right| \nonumber\\
&=& \left|\frac{}{}   \mathbf{P}( -S_n <  -x )+\Phi\left( x \right)-1  \right| \nonumber\\
&=& \left|\frac{}{}   \mathbf{P}( -S_n < -x )  -  \Phi\left( -x \right) \right| \nonumber\\
&\leq&C \, \Big(1+x^2 \Big) \Big( \epsilon |\log \epsilon |  + \frac{\delta}{1+ |x| }  \Big)\exp\left\{ -\frac{\widehat{  x  }^2 }{2} \right\}.
\end{eqnarray}
This completes the proof of Theorem  \ref{th2.1}. \hfill\qed

\vspace{0.3cm}
\noindent\emph{\textbf{Proof of Corollary \ref{co2.1}.}}
Following Bolthausen (1982) we consider the stopping time
$\tau=\sup\{0\leq k\leq n:\ \left\langle S\right\rangle_k \leq 1\}$.
Let $r= \lfloor (1- \left\langle S\right\rangle_\tau)/\epsilon^2 \rfloor$, where $\lfloor x\rfloor$ is the largest integer less than $x$.
Then $r\leq \lfloor 1/\epsilon^2 \rfloor $. Let  $N=n+\lfloor 1/\epsilon^2 \rfloor +1.$
Consider a sequence of independent Rademacher random variables $(\eta_i)$ (taking values $+1$ and $-1$ with equal probabilities)
which is also independent of the martingale differences $(\xi_i)$.
For each $i=1,\dots,N$ define $\xi'_i =\xi_i$ if $i \leq \tau$, $\xi'_i =\varepsilon \eta_i$ if $\tau < i \leq \tau+r,$
$\xi'_{i}=(1-\left\langle S\right\rangle_\tau -r \epsilon^2 )^{1/2}$  if $i=\tau +r +1$, and $\xi'_i=0$ if $\tau +r +1 < i \leq N.$
Clearly, $S'_k=\sum_{i=1}^k \xi'_{i}$, $k=0,\dots,N$ (with $S'_0=0$) is a martingale sequence w.r.t.~the enlarged probability space and the enlarged filtration.
Moreover $\left\langle S'\right\rangle_{N}=1$  and condition (A1) is satisfied for $(S'_k)_{k=1,\dots,N}.$
By Theorem  \ref{th2.1}, it holds, for all $x \in \mathbf{R}$,
\begin{eqnarray}
  \left|\frac{}{} \mathbf{P}(S'_N  \leq x )-  \Phi\left( x\right) \right|   \leq   C\,\left(\frac{}{}\! 1+ x^2 \right)\epsilon|\log \epsilon|   \exp\left\{-\frac{\widehat{x}^2}{2}\right\} ,
\end{eqnarray}
where $\widehat{x}   = \frac{2|x| }{1+\sqrt{1+2|x|\epsilon }}.$  Since $\mathbf{E}[\xi_i^2 | \mathcal{F}_{i-1}] \leq 12 \epsilon^2$ for all $i$ (cf.\ the proof of Lemma 4.1 in \cite{FGL13}), it holds $$ 1- 12 \epsilon^2 \leq 1-\mathbf{E}[\xi_\tau^2 | \mathcal{F}_{\tau-1}] \leq \langle S\rangle_\tau \leq 1.$$ Then it is easy to see that
 \begin{eqnarray}
\ \ \ \ \ \  \mathbf{E}[(S'_N -S_n )^2]  \leq C \, \Big(\mathbf{E} |\! \left\langle S\right\rangle_n -1| + \epsilon^2 \Big);
 \end{eqnarray}
cf.\ Mourrat \cite{M13} for more details.   For all  $x\geq 0$ and  any $t>0,$ it holds
 \begin{eqnarray}
   \mathbf{P}(S_n \leq x ) &\leq&  \frac{}{} \mathbf{P}(S_n \leq x,\ |S_n-S'_N | \leq t  ) +  \mathbf{P}(|S_n-S'_N| > t   )  \label{chi78}\\
    &\leq&\ \mathbf{P}(S'_N \leq x+ t  ) + \frac{1}{t^2} \mathbf{E} [ |S_n-S'_N|^2 ]  \label{chineqd}\\
    &\leq&\ \Phi\left( x+t\right)  + C_1 \left(\frac{}{}\! 1+ x^2 \right)\epsilon|\log \epsilon|   \exp\left\{-\frac{\widehat{x}^2}{2}\right\}
    \nonumber \\ &&\ \ \ \ \ \   + \ \frac{ C_2}{t^2} \Big(\mathbf{E}|\! \left\langle S\right\rangle_n -1| + \epsilon^2 \Big)  \nonumber \\
     &\leq&\ \Phi\left( x \right) +\frac{t}{\sqrt{2\pi}}e^{- x^2/2} \nonumber \\
      && \ \ \ \ \ \  + \  C_1 \left(\frac{}{}\! 1+ x^2 \right)\epsilon|\log \epsilon|   \exp\left\{-\frac{\widehat{x}^2}{2}\right\}  + \frac{ C_2}{t^2}
      \Big(\mathbf{E} |\! \left\langle S\right\rangle_n -1| + \epsilon^2 \Big),  \nonumber
\end{eqnarray}
where $\widehat{x}$ is defined by (\ref{tgn35}).
Now putting $t=(\mathbf{E} |\! \left\langle S\right\rangle_n -1| + \epsilon^2) e^{\, x^2\!/\,6},$ one has, for all $x\geq 0, $
\begin{eqnarray*}
\mathbf{P}(S_n  \leq x )-  \Phi\left( x\right)   & \leq &
C  \bigg( \! \left(  1+ x^2 \right)\epsilon|\log \epsilon| \exp\left\{-\frac{\widehat{x}^2}{2}\right\} \\
 && \ \ \ \ \ \  \ \ \ \    + \ \Big(\mathbf{E} |\! \left\langle S\right\rangle_n -1| + \epsilon^2 \Big)^{1/3}  \exp\left\{-\frac{x^2}{6}\right\} \bigg).
\end{eqnarray*}
It is easy to see that the same bound holds for $ \Phi\left( x\right) -  \mathbf{P}(S_n  \leq x ). $
Therefore we obtain (\ref{thonineqs})
for all $x\geq 0.$ If $x\leq 0, $ we consider $-S_n$ instead of $S_n$ and we use the fact that
$ \mathbf{P}(-S_n \leq -x )= \mathbf{P}(S_n \geq x )= 1- \mathbf{P}(S < x )$.
This implies (\ref{thonineqs}) for all $x \leq 0,$ which completes the proof of
Corollary \ref{co2.1}. \hfill\qed

\section{Proofs of Theorems \ref{th3.1} and \ref{th3.3}}\label{sec5}
The proofs of Theorems \ref{th3.1} and \ref{th3.3} are based on Theorem \ref{th2.1}.

\noindent \emph{\textbf{Proof of Theorem \ref{th3.1}.}} From (\ref{ine29}) and (\ref{ine30}), it is easy to see that
\begin{equation}
\theta_n -\theta = \frac{\sum_{k=1}^n \phi_{k} \varepsilon_k}{\sum_{k=1}^n \phi_{k}^2}. \nonumber
\end{equation}
For any $i=1,...,n$, set
\begin{eqnarray*}
 \xi_i= \frac{ \phi_{i} \varepsilon_i}{ \sigma \sqrt{\sum_{k=1}^n \phi_{k}^2}}\ \  \textrm{and} \ \ \mathcal{F}_{i}' = \sigma \big( \phi_{k}, \varepsilon_k, 1\leq k\leq i,\  \phi_{k}^2, 1\leq k\leq n \big).
\end{eqnarray*}
Then $(\xi_i, \mathcal{F}_{i}')_{i=1,...,n}$ is a sequence of martingale differences which satisfies
$$ \frac{(\theta_n -\theta)\sqrt{\sum_{k=1}^n \phi_{k}^2} } { \sigma}  =\sum_{i=1}^n\xi_i.$$
Notice that$$\langle S\rangle_n = \sum_{i=1}^{n} \frac{ \phi_{i}^2 }{ \sigma^2 (\sum_{k=1}^n \phi_{k}^2)} \mathbf{E}[\varepsilon_i^2 | \mathcal{F}_{i-1} ] =\sum_{i=1}^{n} \frac{ \phi_{i}^2 }{  \sum_{k=1}^n \phi_{k}^2 }  = 1 $$
and
\begin{eqnarray} \label{fgdsasas}
\Big|\mathbf{E}[\xi_i^k | \mathcal{F}_{i-1}' ]\Big| &=&  \frac{ \phi_{i}^2 }{ \sigma^k (\sum_{k=1}^n \phi_{k}^2)} \Bigg|\mathbf{E}\bigg[ \Big(\frac{ \phi_{i}  }{  \sqrt{\sum_{k=1}^n \phi_{k}^2}}\Big)^{k-2} \varepsilon_i^k  \bigg| \mathcal{F}_{i-1}' \bigg]\Bigg| \nonumber\\
&\leq&  \frac{ \phi_{i}^2 \epsilon_1^{k-2} }{ \sigma^k (\sum_{k=1}^n \phi_{k}^2)} \Big|\mathbf{E} \big[   \varepsilon_i^k  \big| \mathcal{F}_{i-1}  \big]\Big| \nonumber\\
&\leq& \frac12 k! \, \epsilon_2^{k-2} \frac{ \phi_{i}^2 \epsilon_1^{k-2} }{ \sigma^{k-2}(\sum_{k=1}^n \phi_{k}^2)} \nonumber\\
&=&  \frac12 k! \, \epsilon ^{k-2} \mathbf{E}[\xi_i^2 | \mathcal{F}_{i-1} ']. \nonumber
 \end{eqnarray}
Applying Theorem
\ref{th2.1} to $(\xi_i, \mathcal{F}_{i}')_{i=1,...,n}$, we obtain the requiblack inequalities. \hfill\qed

\noindent \emph{\textbf{Proof of Theorem \ref{th3.3}.}}
Let $\mathcal{F}_{0} = \sigma \Big( \xi_j^2, 1\leq j\leq n \Big),$ and set, for all $i=1,...,n$,
\begin{eqnarray}\label{dun87}
\eta_i= \frac{ \xi_i}{ \sqrt{[S]_n}}  \ \  \textrm{and}  \ \ \mathcal{F}_{i} = \sigma \Big( \xi_{k},  1\leq k\leq i,\  \xi_j^2, 1\leq j\leq n \Big).
\end{eqnarray}
Since  $(\xi _i)_{i=1,...,n}$ is a  sequence of independent and  symmetric random variables, we deduce that  $$ \mathbf{E}[\xi_i> y|\, \mathcal{F}_{i-1} ] = \mathbf{E}[\xi_i> y|\, \xi_i^2 ] =\mathbf{E}[ -\xi_i > y|\, (-\xi_i)^2 ]=\mathbf{E}[ -\xi_i > y|\, \mathcal{F}_{i-1}  ].$$
Thus $\big(\eta_i,\mathcal{F}_i\big)_{i=1,...,n}$ is a sequence of conditionally symmetric martingale differences, i.e.\ $ \mathbf{E}[\eta_i> y|\, \mathcal{F}_{i-1} ] =  \mathbf{E}[ -\eta_i > y|\, \mathcal{F}_{i-1}  ],$ and satisfies $\mathbf{P}(|\eta_i| \leq \epsilon | \mathcal{F}_{0} )=1$ by assumption.
Note that $S_n / \sqrt{[S]_n} = \sum_{i=1}^n \eta_i
$
and  $\sum_{i=1}^n \mathbf{E}[\eta_i^2| \mathcal{F}_{i-1} ]\  = \ \sum_{i=1}^n\eta_i^2 \ =\ 1.$
By the fact that $\eta_i$ is   conditionally symmetric w.r.t.\ $\mathcal{F}_{i-1}$,   we have, for all $1\leq i\leq n$ and all $\lambda\geq0$,
\begin{eqnarray*}
 \mathbf{E}\left[\exp\left\{\lambda \eta_i  \right\} \Big|  \mathcal{F}_{i-1} \right]  = \frac12 \ \mathbf{E}\left[\exp\left\{  \lambda \eta_i \right\} + \exp\left\{  -\lambda \eta_i \right\} \Big|  \mathcal{F}_{i-1}  \right].
\end{eqnarray*}
Using the inequality $\frac12(e^t+e^{-t}) \leq e^{t^2/2}$, we obtain, for all $\lambda\geq0$,
\begin{eqnarray*}
 \mathbf{E}\left[\exp\left\{ \lambda \eta_i \right\} \Big| \mathcal{F}_{i-1} \right]
  \leq  \exp\left\{\frac{\lambda^2 \eta_i^2}{2  } \right\}.
\end{eqnarray*}
Thus it holds, for all $1\leq k \leq n,$
\begin{equation}
\Psi _k(\lambda ) \ \leq \ \frac{\lambda^2}{2}\sum_{i=1}^k \ \eta_i^2 \  \leq \ \frac{\lambda^2}{2},
\end{equation}
which improves Lemma \ref{lemma2}. With this improvement and a proof similar to that of Theorem
\ref{th2.1}, we obtain the requiblack inequality. \hfill\qed

\appendix \label{assec}

\section{Proof of  Lemma \ref{LEMMA4}}\label{sec6}
 To complete the proof of Theorems \ref{th2.2}, we need to prove Lemma \ref{LEMMA4}.
 We will make use of the following assertion   which gives
 a rate of convergence in the central limit theorem for the
martingale $Y( \lambda  )$ under the conjugate probability measure $\mathbf{P}_{ \lambda  }.$
\begin{lemma}\label{lem3}
Assume conditions (A1) and (A2). Then, for all $0 \leq \lambda < \epsilon^{-1},$
\[
\sup_{u\in \mathbf{R}}\Big| \mathbf{P}_\lambda (Y_n(\lambda )\leq u)-\Phi (u)\Big| \ \leq \
C \Big( \lambda \epsilon +\epsilon \left| \log  \epsilon
 \right| +   \delta \Big) .
\]
\end{lemma}
This assertion is proved in  \cite{FGL13}, Lemma 3.1 (for an earlier result for
bounded martingale differences see Lemma 3.3 of  \cite{GH00}).  Next we use  Lemma \ref{lem3} to prove Lemma \ref{LEMMA4}.

\vspace{0.2cm}

\noindent\emph{\textbf{Proof of  Lemma \ref{LEMMA4}.}}
Notice that $| B_n(\overline{\lambda})-x | \leq C  \overline{\lambda}^2  \epsilon  +  \overline{\lambda} \delta^2 .$ Thus $$\sup_{u\in \mathbf{R}}\Big| \Phi \Big( \frac{\widehat{x}u  - \overline{\lambda} ( B_n(\overline{\lambda})-x   )}{\overline{\lambda} } \Big)  -\Phi ( u)\Big| \ \leq \
C \Big(  \overline{\lambda}^2  \epsilon  +  \overline{\lambda} \delta^2 \Big) .$$ By Lemma \ref{lem3}, it is easy to see that
\begin{eqnarray}
&&\sup_{u\in \mathbf{R}}\Big| \mathbf{P}_{\overline{\lambda}} (  U_n(\overline{\lambda} )   \leq  \widehat{x}u)-\Phi (  u)\Big|  \\
 &\leq &  \sup_{u\in \mathbf{R}}\Big| \mathbf{P}_{\overline{\lambda}} \Big ( Y_n(\overline{\lambda} )  \leq   \frac{\widehat{x}u  - \overline{\lambda} ( B_n(\overline{\lambda})-x   )}{\overline{\lambda} } \Big)-\Phi \Big(  \frac{\widehat{x}u  - \overline{\lambda} ( B_n(\overline{\lambda})-x   )}{\overline{\lambda} } \Big)\Big|  \nonumber \\
& &  + \ \sup_{u\in \mathbf{R}}\Big| \Phi \Big( \frac{\widehat{x}u  - \overline{\lambda} ( B_n(\overline{\lambda})-x   )}{\overline{\lambda} } \Big)  -\Phi ( u)\Big|   \nonumber  \\
&\leq &  C_1 \Big( \overline{\lambda} \epsilon +\epsilon \left| \log  \epsilon
 \right| +   \delta \Big) \, + \, C_2\Big(  \overline{\lambda}^2  \epsilon  +  \overline{\lambda} \delta^2 \Big)  \nonumber  \\
 &\leq &  C  \Big( \overline{\lambda}^2 \epsilon + \overline{\lambda} \delta^2 + \epsilon \left| \log  \epsilon
 \right| +   \delta \Big).   \nonumber
\end{eqnarray}
This completes the proof. \hfill\qed

\section*{Acknowledgements}

\textcolor{black}{The authors are grateful to the reviewers for their comments and remarks which helped to improve the manuscript.}

\end{document}